\documentclass[11pt]{amsart}
\usepackage{epsfig,color}
\usepackage{blindtext}
\usepackage{graphicx}
\usepackage{enumitem}
\usepackage{url}
\usepackage{amssymb}
\usepackage{graphicx,import}
\usepackage{comment}
\usepackage{esint}
\usepackage{xypic}
\usepackage{fancyhdr}

\usepackage[margin = 1.4in] {geometry}

\usepackage{hyperref}

\numberwithin{equation}{section}

\theoremstyle{theorem}

\newtheorem{theorem}{Theorem}[section]

\theoremstyle{definition}

\newtheorem*{remark*}{Remark}

\numberwithin{equation}{section}

\title{Rigidity of Minimal Surfaces with Y-singularities and low Morse Index  in $\mathbb{R}^3$}

\author{Elham Matinpour}
\address{Department of Mathematics, Johns Hopkins University, 3400 N. Charles Street, Baltimore, MD 21218}
\email{ematinp1@jhu.edu}

\begin{document}

\maketitle

\begin{abstract}  
We investigate the geometric constraints imposed by low Morse index on minimal surfaces with $Y$-singularities, focusing on the classification of those with Morse index one. Our rigidity result establishes a partial uniqueness theorem, highlighting the $Y$-catenoid as a distinguished example among complete, two-sided minimal surfaces in $\mathbb{R}^3$ with $Y$-singularities and Morse index one.
\end{abstract}

\section{introduction}

A central problem in the field of minimal surfaces concerns the different ways we can deform a minimal surface to decrease its area to the second order. The Morse index quantifies this problem. The next interesting problem concerns identifying the minimal surfaces which exhibit the lowest Morse indices, and what these indices reveal about their geometric and analytic structure. In $\mathbb{R}^3$, significant progress has been made in classifying minimal surfaces with low Morse index. Fischer-Colbrie and Schoen \cite{fischer1980structure} famously proved that any complete, stable minimal surface characterized by a Morse index of zero must be a plane, underscoring the rigidity of stability. Beyond this trivial case, the classification of surfaces with positive but low index has remained a vibrant area of inquiry. Lopez and Ros \cite{lopez1989complete} demonstrated that the catenoid and Enneper’s surface are the only complete, oriented minimal surfaces in $\mathbb{R}^3$  with index one, while Chodosh and Maximo \cite{chodosh2023topology} established the non-existence of complete two-sided minimal surfaces with index two. For index three, notable examples of two-sided immersed minimal surfaces include the Chen-Gackstatter surface [Corollary 15, \cite{montiel1991schrodinger}], the Jorge-Meeks trinoid, the generalized Enneper surface with a multiplicity-five end, and the Richmond surface \cite{tuzhilin1992morse}, each illustrating the increasing complexity of higher-index configurations.
\medskip

While the theory of smooth minimal surfaces is well-developed, the analysis of non-smooth minimal surfaces, particularly those with singularities, presents fresh challenges and opportunities. In this work, we shift focus to minimal surfaces in $\mathbb{R}^3$ exhibiting $Y$-type singularities, a class of surfaces whose irregular geometry complicates traditional stability analysis. A natural extension of the classical question arises: Which singular minimal surfaces possess the lowest Morse indices, and how do their singularities influence their variational properties? Recent advances provide a foundation for our investigation. Wang \cite{wang2022curvature} showed that complete $Y$-singular minimal surfaces with quadratic area growth and Morse index zero are composed of flat surface components, mirroring the planar rigidity of the smooth case. Furthermore, \cite{matinpour2024morse} established that the $Y$-catenoid has Morse index one, while a family of $Y$-noids exhibits index two, suggesting a rich hierarchy of stability among singular surfaces. Building on these insights, this paper advances the classification of $Y$-singular minimal surfaces by proving a partial uniqueness of the $Y$-catenoid among the minimal $Y$-surfaces in $\mathbb{R}^3$ with Morse index one, Theorem \ref{thm: main thm}.

\subsection*{Acknowledgment}
I wish to express my gratitude to my advisor, Jacob Bernstein, for proposing this problem and encouraging me to tackle it, and also for his helpful comments and advice. 

\section{Preliminary Notations}
We define a $Y$-surface $\Sigma \subset \mathbb{R}^3$ as a rectifiable current mod 3 such that the support of $\Sigma$ is $C^{1,\alpha}_{loc}$ diffeomorphic to a plane, a half-plane, or a $Y$-cone, where $Y$-cone is a union of three half-planes joined along the boundary line forming a $Y$-configuration, that means they make an angle equal to 120-degree to each other along the common junction. Here, $C_{loc}^{1,\alpha}$ is the standard Holder space of $C^{1,\alpha}$-regular diffeomorphisms, for $\alpha <1$, defined on open neighborhoods of $\Sigma$. This means that for each point $p\in \Sigma$ we can find a neighborhood $N(p)\subset \Sigma$ such that $N(p)$ is $C^{1,\alpha}$ diffeomorphic to one of the three cones: a plane, a half-plane, a $Y$-cone. For $\Sigma$ to be a $Y$-surface, it further needs to satisfy a tangential $Y$-configuration. That is, the tangent space $T_p\Sigma$ must be one of these three cones, in Hausdorff sense. This requires that the differential $Du(p)$ of the corresponding map $u\in C_{loc}^{1,\alpha}$ be in some orthogonal groups $O(3)$ of $\mathbb{R}^3$. Hence, the tangent cone $T_p\Sigma$ to $\Sigma$ at $p$ is always one of the cones, a plane, a half-plane, a $Y$-cone. We are only concerned with minimal $Y$-surfaces, so we further assume that $\Sigma$ is minimal, that is, $\Sigma$ is a stationary current. Here, a stationary current has a (weak) vanishing mean curvature at the regular points, the points locally diffeomorphic to the plane. 
\medskip

We also need $\Sigma$ to be two-sided, in the sense that it is separating in $\mathbb{R}^3$. We say that $\Sigma$ defines a cell structure in the open set $U\subset \mathbb{R}^3$ if there exists a family of open, connected sets $\mathcal{C}(\Sigma)=\{ U^i:\ 1\leq i\leq N\}$, called the cells of $\Sigma$, such that
$$
\partial \Sigma \subset \partial U,\ \ \ \ \ U\setminus \Sigma =\bigcup_{i=1}^N
U^i$$
and for each $p\in \Sigma \cap U=\Sigma \setminus \partial \Sigma$ there is a $r>0$ such that $B_r(p)\subset U$, and for each $0<r'<r$ and $i=1,\cdots,N$, $B_{r'}(p)\cap U^i$ is connected (possibly empty), where $B_r(p)$ is the ball with radius $r$ and centered at $p$ in $\mathbb{R}^3$. In this note, the $Y$-surface $\Sigma$ is two-sided, which means that $\Sigma$ defines a cell structure in $\mathbb{R}^3$, and thus separates in $\mathbb{R}^3$.\\ 
The $Y$-surface in our main theorem is a complete surface without boundary, so the set of points locally diffeomorphic to the half-plane is empty. Here, with a complete current we mean the current $\Sigma$ without boundary that has a geodesically complete support; example of a complete support is a plane. We need $\Sigma$ to have a Euclidean area growth. Define the upper density of $\Sigma$ at $p$ to be 
$$
\bar{\Theta}(\Sigma ,p)=\limsup_{r\to 0^+} \frac{\mathcal{H}^2(\Sigma \cap B_r(p))}{\pi r^2}
$$
When the usual limit exists, we denote it by $\Theta(\Sigma ,p)$ and call it the density of $\Sigma$ at $p$. Then for $\Sigma$ to be a $Y$-surface, $\bar{\Theta}(\Sigma,p)\leq 3/2$. We assume that $\Sigma$ has a Euclidean area growth which means that $\bar{\Theta}(\Sigma ,p)$ is uniformly bounded above by a finite number, where $p$ can be infinity, and the density of $\Sigma$ at infinity is
$$
\bar{\Theta}(\Sigma ,\infty)=\limsup_{r\to \infty} \frac{\mathcal{H}^2(\Sigma \cap B_r(0))}{\pi r^2}
$$
Precisely, $\Sigma$ has a Euclidean area growth if there exists a constant $C>0$ such that for any $r>0$ and any $p\in \Sigma$, we have
$$
\mathcal{H}^2(\Sigma \cap B_r(p)) \leq Cr^2
$$
In this note, $\Sigma$ is a minimal $Y$-surface, as specified above, a stationary rectifiable current mod $3$ with a complete support, $C_{loc}^{1,\alpha}$-diffeomorphic to a plane or a $Y$-cone, two-sided, and has a Euclidean area growth. Let $\Gamma$ to be the set of the points $p\in \Sigma$ which are locally diffeomorphic to the $Y$-cone, we first prove the theorem for $\Gamma$ to be a connected closed curve. Since $\Sigma$ is a mod 3 current, we can write $\Sigma$ as a union of some integer currents $\Sigma_i$ (with multiplicity one), if $\Gamma$ is connected then $i\in \{1,2,3\}$, so that each $\Sigma_i$ satisfies the conditions specified above. We use decomposition of the support in the definition of $\Sigma$, and so we refer to these integer currents when we are computing the Morse index of $\Sigma$, this is because any bounded negative deformation on each of the integer currents will add to the index, and so we need to treat each of them separately at some point. To make it clear we define
$$
\text{reg}(\Sigma):=\{p\in \Sigma,\ T_p\Sigma \ \text{is a plane}\},    
$$
$$
 \mathcal{R}:=\{ \Sigma_i\subset \text{reg}(\Sigma),\ \Sigma_i\ \text{is a connected component of reg}(\Sigma)\}, 
$$
and,
$$
\Sigma_Y:=\{ p\in \Sigma,\ T_p\Sigma \ \text{is a $Y$-cone}\},   
$$
$$
\mathcal{Y}:=\{ \Gamma \subset \Sigma_Y,\  \Gamma \ \text{is a connected component of}\ \Sigma_Y\}
$$
We may define, $B(\Sigma):=\{ p\in \Sigma,\ T_p\Sigma \ \text{is a half-plane}\}$, in this note $B(\Sigma)=\emptyset$.
To discuss the index of $\Sigma$, we need to fix some notion of the unit normal vector field on $\Sigma$, the second fundamental form on $\Sigma$, and the mean curvature of $\Sigma$. Set $\nu=(\nu_{\Sigma_i})_{\Sigma_i\in \mathcal{R}}$ to be a choice of unit normal vector field on $\Sigma$ that is compatible with the $Y$-configuration on $\Sigma$, where $\nu_{\Sigma_i}$ is a unit normal vector field on $\Sigma_i \in \mathcal{R}$. Let $A_{\Sigma_i}$ to be the second fundamental form on $\Sigma_i$, and define $A=(A_{\Sigma_i})_{\Sigma_i\in \mathcal{R}}$ respectively. Let $H=(H_{\Sigma_i})_{\Sigma_i\in \mathcal{R}}$ denote the mean curvature of $\Sigma$ with respect to the unit normal $\nu$, where $H_{\Sigma_i}$ is the mean curvature of $\Sigma_i$.
\medskip

Let $\varphi_t$ to be a family of compactly supported $C_c^{1,\alpha}$ diffeomorphisms in $\mathbb{R}^3$ with $\varphi_0$ being the identity map. Set $\Sigma_t=\varphi_t(\Sigma)$ for $t\in (-\epsilon,\epsilon)$, then $\Sigma_t$ is a compactly supported variation of $\Sigma$. Write $V(x):=\frac{d}{dt}\vert_{t=0}\varphi_t(x)$, the variational vector field on $\Sigma$ associated with the variation $\Sigma_t$, where $V\in C_c^{1,\alpha}(\Sigma,\mathbb{R}^3)$ is the associated vector field in $\mathbb{R}^3$. We may define the normal variation on $\Sigma$ induced by $V$ as 
$$
f:=(f_{\Sigma_i})_{\Sigma_i\in \mathcal{R}},\ f_{\Sigma_i}=V\cdot \nu_{\Sigma_i},\ \Sigma_i \in \mathcal{R}
$$
We require the functions defined on $\Sigma$ to satisfy a compatible condition with the $Y$-configuration of the $Y$-surfaces. That is, $f\in C_c^{1,\alpha}$ satisfies the compatible condition if $f$ can be written as $f=V\cdot \nu$ for some $V\in C_c^{1,\alpha}(\Sigma)$. This condition makes the induced functions on $\Gamma \in \mathcal{Y}$ sum up to zero. Namely, $f\in C_c^{1,\alpha}(\Sigma)$ satisfies the compatible condition is equivalent to 
$$
\underset{i=1}{\overset{3}{\sum}}f_i\vert_{\Gamma}(x)=0,\ x\in \Gamma,\  \Gamma\in \mathcal{Y}
$$
To make it more compatible with the index form, we consider the following Sobolev space
$$
W_{com}^{1,2}(\Sigma)=\{ f=(f_{i})_{\Sigma_i\in \mathcal{R}},\ f_i\in W_c^{1,2}(\Sigma_i),\ \underset{i=1}{\overset{3}{\sum}}f_i\vert_{\Gamma}(x)=0,\ x\in \Gamma,\  \Gamma\in \mathcal{Y}\} 
$$
The second variation of area of $\Sigma_t$ associated to the normal variation $V$ at $\Sigma =\Sigma(0)$ can be  written as the quadratic form
\begin{equation}
    \nonumber
    Q : W_{com}^{1,2} (\Sigma) \rightarrow \mathbb{R} 
\end{equation}
given by 
\begin{equation} \label{eqn: 2nd variation fromula}
\begin{split}
   Q(f , f) = \frac{d^2}{dt^2} Area (\Sigma (t))\vert_{t=0} =  \underset{\Sigma_j\in \mathcal{R}}{\overset{}{\sum}}  \int_{\Sigma_j} \vert \nabla_{\Sigma_j} f_j \vert^2 - \vert A_{\Sigma_j} \vert^2 (f_j )^2  
     - \underset{\Gamma\in \mathcal{Y}}{\overset{}{\sum}}\underset{i=1}{\overset{3}{\sum}} \int_{\Gamma }    \bold{H}_{\Gamma} \cdot \tau_i (f_i \vert_{\Gamma} )^2 ,
   \end{split}
\end{equation}
where $\vert A_{\Sigma_j} \vert^2$ is the norm square of the second fundamental form of $\Sigma_j$, and $\bold{H}_{\Gamma}$ is the mean curvature vector or the geodesic curvature vector of the curve $\Gamma$ in $\mathbb{R}^3$. 
\medskip

\label{def: Morse Index} The Morse index of $\Sigma \subset \mathbb{R}^3$ is the maximal dimension of a subspace of $ W_{com}^{1,2} (\Sigma)$ on which the second variation of area functional, the index form $Q$, is negative, and the nullity is the dimension of the kernel of $Q$, \cite{wang2022curvature}, \cite{matinpour2024morse}.


\section{Uniqueness of the $Y$-surfaces with the Index One}\label{sec: main result}
\begin{theorem} \label{thm: main thm}
          Suppose that $\Sigma$ is a complete, two-sided, minimal $Y$-surface in $\mathbb{R}^3$. Suppose that $\Sigma$ has three faces, $\Sigma_1$, $\Sigma_2$, and $\Sigma_3$, and it has the compact interface $\Gamma$. If the Morse index of $\Sigma$ is equal to one then up to relabeling the faces:\\
          $\Sigma_3$ is a compact disk with the total curvature at most $2\pi$, $\Sigma_2$ is an unbounded annulus with an embedded end and total curvature at most $2\pi$, and $\Sigma_1$ is unbounded and (Dirichlet) stable.
          In addition, if $\Sigma_1$ has only one end then $\Sigma$ is a $Y$-catenoid.
        
\end{theorem}

\begin{proof}
We assume that $\Sigma$ has a Euclidean volume growth, this is a natural assumption, in particular, on surfaces $\Sigma$ with finite total curvature in $\mathbb{R}^3$,  [Theorem 2. \cite{fischer1985complete}]. The domain of the index form is slightly larger than $W_c^{1,2}(\Sigma)$-space, in fact, it contains the space of $L^2(\Sigma)$ functions because they support a cut-off function to become compactly supported functions. However, the set of allowable functions may be slightly bigger than $L^2(\Sigma)$-space. Chodosh and Maximo, \cite{chodosh2023topology}, introduced a weighted $L^2$-space, denoted by $L_*^2$, designed to maximize the number of $L^2$-orthogonal functions in the domain of the index form, such that they remain sufficiently controlled to establish a weighted analogue of Fischer-Colbrie’s index result \cite{fischer1985complete} and support a cut-off function argument (see \cite{chodosh2023topology} for details). This space $L_*^2(\Sigma)$ is defined as the completion of smooth, compactly supported functions with respect to the norm,
$$
\Vert f\Vert_{L_*^2(\Sigma)}^{2}:=\int_{\Sigma}f^2(1+\vert x\vert^2)^{-1}(\log(2+\vert x\vert))^{-2},
$$
where $\vert x\vert$ is the Euclidean distance. \\
Consider the minimal $Y$-surface $\Sigma =(\Sigma_j\in \mathcal{R}; \Gamma \in \mathcal{Y})$, we first assume that $\Gamma$ is a single closed curve. Then there are three faces $\Sigma_j \in\mathcal{R}$ associated with $\Gamma$, and we can denote $\Sigma$ briefly by $\Sigma =(\underset{j=1}{\overset{3}{\cup }} \Sigma_j; \Gamma )$. Then it follows from the fact that the only singular points are the $Y$-points that $\Gamma$ is embedded, and thus a simple closed curve. The finite index implies that the total curvature of $\Sigma$, and thus of each $\Sigma_j$, is finite. Then each $\Sigma_j$ has a Euclidean area growth, and then we can check that a constant function, say $1$, is in the $L^2_*(\Sigma)$-space. That can be checked by an estimation:
$$
\int_1^{\infty}\frac{1}{1+r^2}\frac{1}{(\log(2+ r))^2}rdr <\infty .
$$
Therefore, we can plug a constant function into the index form after applying some cut-off functions. Let $\phi=(c_1,c_2,c_3)$ be a function on $\Sigma$ that satisfies the compatibility condition $c_1+c_2+c_3=0$, and $c_i$s are some real constants. Let $\xi :\mathbb{R}\to \mathbb{R}$ be a smooth function satisfying $\xi(s)=0$ for $s\leq 0$ and $\xi (s)=1$ for $s\geq 1$. Then there exists a constant $C>0$ such that $\vert \xi'\vert,\vert \xi''\vert\leq C$. For any $R>1$, we define $\varphi_R:\Sigma\to \mathbb{R}$ to be:
$$
\varphi_R(x)=\xi\circ \psi_R(x),
$$
where $\psi_R(x)=2-\frac{\log\vert x\vert}{\log R}$ is the standard log-cutoff function. Then one can check that
$$
\vert \nabla \varphi_R\vert \leq \frac{C}{\vert x\vert \log R},\ \ \ \vert \Delta \varphi_R \vert \leq \frac{C}{\vert x\vert^2 \log^2R}+\frac{C}{\vert x\vert^{2+2/k}\log R},
$$
see \cite{chodosh2023topology} for more details. 
Let $f=\phi \varphi_R$, and then compute
\begin{equation}
\nonumber
\begin{split}
Q(f)&=\underset{i=1}{\overset{3}{\sum}}(\int_{\Sigma_i}\vert \nabla f_i\vert^2 -\vert A_i\vert^2f_i^2 -\int_{\Gamma} \bold{H}_{\Gamma}\cdot \tau_i f_i^2)\\
&=\underset{i=1}{\overset{3}{\sum}}(\int_{\Sigma_i\cap B_R(0)} -\vert A_i\vert^2c_i^2 +\int_{B_{2R}(0)\setminus B_R(0)} c_i^2\vert \nabla \varphi_R \vert^2 -\vert A_i\vert^2c_i^2\varphi_R^2 -\int_{\Gamma} \bold{H}_{\Gamma}\cdot \tau_i c_i^2)\\
&=\underset{i=1}{\overset{3}{\sum}}(\int_{\Sigma_i\cap B_R(0)} -\vert A_i\vert^2c_i^2  -\int_{\Gamma} \bold{H}_{\Gamma}\cdot \tau_i c_i^2\\
&+\int_{B_{2R}(0)\setminus B_R(0)} c_i^2\vert \nabla \varphi_R \vert^2 -\vert A_i\vert^2c_i^2\varphi_R^2)
\end{split}
\end{equation}
For $R>0$ sufficiently large $\int_{B_{2R}(0)\setminus B_R(0)}\vert A_i\vert^2 \approx 0$, and we can estimate
$$
\int_{B_{2R}(0)\setminus B_R(0)} \vert \nabla \varphi_R \vert^2 -\vert A_i\vert^2 \varphi_R^2 \approx \int_{B_{2R}(0)\setminus B_R(0)} \vert \nabla \varphi_R \vert^2 \approx 0
$$
Hence,
$$
Q(f)\approx \underset{i=1}{\overset{3}{\sum}}(\int_{\Sigma_i\cap B_R(0)} -\vert A_i\vert^2c_i^2  -\int_{\Gamma} \bold{H}_{\Gamma}\cdot \tau_i c_i^2)
$$
Suppose that each face $\Sigma_i$ has a finite number of ends $e_i$ with finite multiplicities summing up to $d_i$, and let $g_i$ to be the number of genus of $\Sigma_i$. By the Gauss-Bonnet theorem,
$\int_{\Sigma_i}K_idA+\int_{\Gamma}k_ids + 2\pi d_i=2\pi \chi (\Sigma_i)$, where $K_i$ is the Gaussian curvature of $\Sigma_i$ and $k_i$ is the curvature of $\Gamma$ relative to $\Sigma_i$. Recall that,
$\chi(\Sigma_i )=2-2g_i-1-e_i=1-2g_i-e_i$ of $\Sigma_i$, and we have $\vert A_i\vert^2 =-2K_i$ on minimal surfaces, then compute
\begin{equation}
\nonumber
\begin{split}
Q(f)&\approx \underset{i=1}{\overset{3}{\sum}}(\int_{\Sigma_i\cap B_R(0)} -\vert A_i\vert^2c_i^2  -\int_{\Gamma} \bold{H}_{\Gamma}\cdot \tau_i c_i^2)\\
&\approx \underset{i=1}{\overset{3}{\sum}}(c_i^2 (2\pi (1-2g_i-e_i-d_i))-\int_{\Sigma_i\cap B_R(0)} \frac{1}{2}\vert A_i\vert^2 +c_i^2\int_{\Gamma} -k_i -c_i^2\int_{\Gamma} -k_i )\\
&\approx \underset{i=1}{\overset{3}{\sum}}(c_i^2 (2\pi (1-2g_i-e_i-d_i))-\int_{\Sigma_i\cap B_R(0)} \frac{1}{2}\vert A_i\vert^2 )
\end{split}
\end{equation}
where $k_i=-\bold{H}_{\Gamma}\cdot \tau_i$. 
Hence,
\begin{equation}\label{eqn: Q(f), f=(c1,c2,c3)}
Q(f)\approx \underset{i=1}{\overset{3}{\sum}} c_i^2  (2\pi (1-2g_i-e_i-d_i))-\int_{\Sigma_i\cap B_R(0)}\frac{1}{2}\vert A_i\vert^2 
\end{equation}
Write $\alpha_i=2\pi(1-2g_i-e_i-d_i)$ and $\beta_i=-\frac{1}{2}\int_{\Sigma_i}\vert A_i\vert^2$. Then the sign of the index form would be determined by the sign of $\alpha_i+\beta_i$s. Clearly, if $\alpha_i+\beta_i<0$ for all three faces then it follows from the compatibility condition of $Y$-surface that the index is at least two. Therefore, we need to have $\alpha_i+\beta_i\geq 0$ for some $i$. \\
Suppose that $\alpha_i+\beta_i=0$ on $\Sigma_i$. It follows from $\beta_i \leq 0$ that $\alpha_i\geq 0$, and thus $g_i=0$ and $e_i=0$. Then $\alpha_i=2\pi$ and $\Sigma_i$ is a bounded disk.  \\
Now suppose that  $\alpha_i+\beta_i>0$ on $\Sigma_i$, then from $\beta_i\leq 0$ we have $\alpha_i >0$, and thus $g_i=0$ and $e_i=0$. Thus $\Sigma_i$ is a bounded disk.\\

Set $\theta_i:=\alpha_i+\beta_i$, and order $\theta_1 \leq \theta_2 \leq \theta_3$. Noting that $\theta_i<0$ is possible for at most two faces due to the index restriction, conclude that $\theta_3 \geq 0$, then summarize the possible cases as follows.\\
 Suppose that $\theta_3\geq0$, then $\Sigma_3$ is a bounded disk, say a compact disk. Note that, it is not possible for all three faces to be compact because then by the convex hull property for compact minimal surfaces, three compact faces are enclosed by the convex hull of $\Gamma$, and so $\Sigma$ fails to keep the $Y$-configuration, see \cite{colding2011course}, \cite{osserman1971convex}, for a reference on the convex hull property. Note that the structure of two compact faces and one non-flat unbounded face is not possible by the strong half-space theorem. Therefore, there must be a compact face, and the two other faces are unbounded. We may check that it is not possible for an unbounded face to be a punctured plane; otherwise, the compact face must be flat too. This is because $\Gamma$ is then a planar curve, and the minimal compact face then makes a constant angle to the plane along $\Gamma$, and thus must be a flat disk, \cite{nitsche1989lectures}, \cite{nitsche1965new}, which is impossible due to the $Y$-configuration of $\Sigma$. \\
We conclude that $\theta_1, \theta_2<0$ and $\theta_3\geq0$ so that the possible structures are one compact and the two other non-flat unbounded faces. 
\medskip

Note that since $\theta_3\leq 2\pi$, if both $\theta_2$ and $\theta_1$ are large negative numbers, then the quadratic index form has two negative directions, which makes the index at least two. We may represent the index form by a $3\times 3$ matrix and then write it diagonally with $\theta_i$s on the diagonal. Then employing the compatible condition, $c_1+c_2+c_3=0$, we get a matrix $2\times 2$ representing the index form with its trace equal to $\theta_1+\theta_2+2\theta_3$ and its determinant equal to $\theta_1\theta_2+\theta_3(\theta_1+\theta_2)$. It can be checked that if the trace is negative while the determinant is positive, then there are two negative directions for the index form. Now suppose that $\theta_2 \leq -4\pi$, then by $\theta_1\leq \theta_2$ we have $\theta_1 \leq -4\pi$. Then since $0\leq \theta_3\leq 2\pi$ we find that the trace of the matrix associated with the index form is negative while its determinant is positive, which implies that the index of $\Sigma$ is at least two, a contradiction. Hence, we may require $-4\pi \leq \theta_2\leq -2\pi$ which implies that $\Sigma_2$ is an annulus. Note that $\theta_2=-2\pi$ implies that $\Sigma_2$ is flat, which is not possible, so we have $-4\pi \leq \theta_2< -2\pi$. Hence, $\Sigma$ consists of a compact disk $\Sigma_3$, an unbounded annulus $\Sigma_2$, and a third unbounded face $\Sigma_1$. 
\medskip

Suppose that $\theta_3 =2\pi$ then $\Sigma_3$ is a flat compact disk and thus $\Sigma_2$ and $\Sigma_1$ meet the plane along $\Gamma$ at an equal and constant angle; hence, by the Hopf lemma they are symmetric. Hence, $\Sigma_1$ is also an annulus. 

Now suppose that $\theta_3< 2\pi$. A simple computation shows that if $\theta_3 <\pi$, that is equivalent to $-\int_{\Sigma_3}\vert A_3\vert^2 <-2\pi$, then\\
$$
\theta_1+\theta_2+2\theta_3<0, \ \ \ \theta_1\theta_2+\theta_3(\theta_1+\theta_2)>0
$$
which implies that the index is at least two. 
Therefore, we must have $\pi \leq \theta_3 < 2\pi$, and then $-4\pi \leq \theta_2 < -2\pi$. If $-4\pi\leq \theta_1<-2\pi$ then $\Sigma_1$ is also an annulus. If $\Sigma_1$ is not an annulus, then we may assume that $\Sigma_1$ has more than one end, or it has some genus. This results in $\theta_1 \leq-6\pi$.
Then we can check $-3\pi \leq \theta_2 <-2\pi$. Suppose that $\theta_1= -6\pi$ then $\theta_1\theta_2+\theta_3(\theta_1+\theta_2)= -6\pi \theta_2 +\theta_3 (-6\pi +\theta_2)>-6\pi \theta_2 + 2\pi (-6\pi +\theta_2)=-4\pi \theta_2-12\pi^2$, this is positive if $\theta_2< -3\pi$, and this makes the index form negative on a space of dimension at least two, a contradiction. Therefore, if $\Sigma_1$ is not an annulus, then $\Sigma_2$ is an annulus with the total curvature $\int_{\Sigma_2}\vert A_2\vert^2 \leq 2\pi$. If $\Sigma_1$ is an annulus or if $\Sigma_1$ has only one end, then the result of Bernstein and Maggi [Theorem 1.4, \cite{bernstein2021symmetry}] applies and $\Sigma$ must then be a $Y$-catenoid. Otherwise, $\Sigma_3$ is a disk with the total curvature at most $2\pi$, this would be directly checked by $\pi\leq \theta_3<2\pi$, $\Sigma_2$ is an annulus with the total curvature at most $2\pi$, and $\Sigma_1$ unbounded and stable under the Dirichlet condition.
\end{proof}



\begin{thebibliography}{99}
\bibitem[Ber21]{bernstein2021symmetry}
Jacob Bernstein and Francesco Maggi, \emph{Symmetry and rigidity of minimal surfaces with plateau-like singularities}, Archive for Rational Mechanics and Analysis \textbf{239} (2021), no.~2, 1177--1210.

\bibitem[Cho23]{chodosh2023topology}
Otis Chodosh and Davi Maximo, \emph{On the topology and index of minimal surfaces II}, Journal of Differential Geometry \textbf{123} (2023), no.~3, 431--459.

\bibitem[Col11]{colding2011course}
Tobias H. Colding and William P. Minicozzi, \emph{A course in minimal surfaces}, vol. 121, American Mathematical Soc., 2011.

\bibitem[Fis80]{fischer1980structure}
Doris Fischer-Colbrie and Richard Schoen, \emph{The structure of complete stable minimal surfaces in 3-manifolds of non-negative scalar curvature}, Communications on Pure and Applied Mathematics \textbf{33} (1980), no.~2, 199--211.
\bibitem[Fis85]{fischer1985complete}
Doris Fischer-Colbrie, \emph{On complete minimal surfaces with finite Morse index in three manifolds}, Inventiones mathematicae \textbf{82} (1985), no.~1, 121--132.

\bibitem[Lóp89]{lopez1989complete}
Francisco J. López and Antonio Ros, \emph{Complete minimal surfaces with index one and stable constant mean curvature surfaces}, Comment. Math. Helv. \textbf{64} (1989), no.~1, 34--43.

\bibitem[Mat24]{matinpour2024morse}
Elham Matinpour, \emph{Morse index of y-singular minimal surfaces}, arXiv preprint arXiv:2410.15534 (2024).

\bibitem[Mon91]{montiel1991schrodinger}
Sebastián Montiel and Antonio Ros, \emph{Schrödinger operators associated to a holomorphic map}, Global differential geometry and global analysis, Springer, 147--174, 1991.

\bibitem[Nit65]{nitsche1965new}
Johannes C. C. Nitsche, \emph{On new results in the theory of minimal surfaces}, Bulletin of the American Mathematical Society \textbf{71} (1965), no.~2, 195--270.

\bibitem[Nit89]{nitsche1989lectures}
Johannes C. C. Nitsche, \emph{Lectures on minimal surfaces: vol. 1}, Cambridge University Press, 1989.

\bibitem[Oss71]{osserman1971convex}
Robert Osserman, \emph{The convex hull property of immersed manifolds}, Journal of Differential Geometry \textbf{6} (1971), no.~2, 267--270.

\bibitem[Tuz92]{tuzhilin1992morse}
Alexey A. Tuzhilin, \emph{Morse-type indices of two-dimensional minimal surfaces in $R^3$ and $H^3$}, Mathematics of the USSR-Izvestiya \textbf{38} (1992), no.~3, 575.

\bibitem[Wan22a]{wang2022curvature}
Gaoming Wang, \emph{Curvature estimates for stable minimal surfaces with a common free boundary}, Calculus of Variations and Partial Differential Equations \textbf{61} (2022), no.~4, 1--26.

\end{thebibliography}
\end{document}